\documentclass[12pt,a4paper]{amsart}

\usepackage{amssymb,color,currfile,pdfpages,mathtools}

\usepackage[english]{babel}
\usepackage{verbatim,here}
\usepackage[T1]{fontenc}
\usepackage{epstopdf}
\usepackage{floatflt,graphicx}
\usepackage{color,xcolor}
\usepackage{hyperref}
\usepackage{enumerate}
\usepackage{animate}
\usepackage{a4wide}
\usepackage{amsmath, amssymb}
\usepackage{amsthm}
\usepackage{float}
\usepackage{pdfpages}
\usepackage{algorithm}
\usepackage{algpseudocode}
\newcounter{minutes}
\divide\time by 60
\newcounter{hours}
\multiply\time by 60 \addtocounter{minutes}{-\time}

\dedicatory{}
\commby{}
\theoremstyle{plain}
\newtheorem{thm}[equation]{Theorem}
\newtheorem{cor}[equation]{Corollary}
\newtheorem{lem}[equation]{Lemma}

\theoremstyle{definition}

\theoremstyle{remark}
\newtheorem{rem}[equation]{Remark}

\newtheorem{nonsec}[equation]{}

\numberwithin{equation}{section}

\newcommand{\beq}{\begin{equation}}
\newcommand{\eeq}{\end{equation}}
\newcommand{\ben}{\begin{enumerate}}
\newcommand{\een}{\end{enumerate}}
\newcommand{\bequu}{\begin{eqnarray*}}
\newcommand{\eequu}{\end{eqnarray*}}
\newcommand{\bequ}{\begin{eqnarray}}
\newcommand{\eequ}{\end{eqnarray}}

\newcommand{\sh}{\,\textnormal{sh}}
\newcommand{\ch}{\,\textnormal{ch}}
\renewcommand{\th}{\,\textnormal{th}}

\begin{document}
\thispagestyle{empty}
\def\thefootnote{}



\title[Hilbert metric and quasiconformal mappings]{Hilbert metric and quasiconformal mappings}

\author[\c{S}. Alt\i nkaya]{\c{S}ahsene Alt\i nkaya}
\address{Department of Mathematics and Statistics,
	University of Turku,\newline FI-20014 Turku, 
	Finland\\
	\url{https://orcid.org/0000-0002-7950-8450}}
\email{sahsene.altinkaya@utu.fi}

\author[M. FUJIMURA]{Masayo FUJIMURA}
\address{Department of Mathematics,
	National Defense Academy of Japan, \newline Yokosuka, Japan \\
	\url{https://orcid.org/0000-0002-5837-8167}}
\email{masayo@nda.ac.jp}

\author[M. Vuorinen]{Matti Vuorinen}
\address{Department of Mathematics and Statistics,
	University of Turku, \newline
	FI-20014 Turku,
	Finland\\ \url{https://orcid.org/0000-0002-1734-8228}} 
\email{vuorinen@utu.fi}

%

\date{}

\begin{abstract}
We prove a functional identity between the Hilbert metric and the visual angle metric
in the unit disk. The proof utilizes the Poincar\'e hyperbolic metric in terms of
which both metrics can be expressed. This identity then yields sharp distortion results
for quasiregular mappings and analytic functions, expressed in terms of the Hilbert
metric. We also prove that Hilbert circles are, in fact,
Euclidean ellipses. The proof makes use of computer algebra methods. In particular,
Gr\"obner bases are used.

\end{abstract}

\footnote{\c{S}. Alt\i nkaya is supported by the Scientific and Technological Research Council of T\"{u}rkiye (TUBITAK 2219 - International Postdoctoral Research Fellowship Program for Turkish Citizens), Project Number: 1059B192402218.
M. Fujimura is partially supported by JSPS KAKENHI Grant Number JP25K07039.}

\maketitle


\footnotetext{\texttt{{\tiny File:~\jobname .tex, printed: \number\year-%
\number\month-\number\day, \thehours.\ifnum\theminutes<10{0}\fi\theminutes}}}
\makeatletter

\makeatother



\section{Introduction}
In recent years, hyperbolic metrics and metrics  similar to it have
become standard tools of geometric function theory  \cite{dhv,frv,gh,hkv,ha,h}. 
In his work \cite[pp.42-48]{p}, the author comprehensively lists twelve metrics that frequently occur in complex analysis, underscoring their significance in this field. These metrics, often referred to as hyperbolic-type metrics, are generally not M\"obius invariant; however, they are frequently quasi-invariant and differ from the hyperbolic metric at most by a constant factor.

In this paper, we apply these ideas to prove a new functional identity 
for the Hilbert metric. This metric, closely related to the Klein or
Cayley-Klein metrics, is studied in \cite{b2,p,py1,py2,pt,rv, sop}. For all distinct points $ a $ and $ b $ in the unit disk $  \mathbb{B}^2 $, the {\it Hilbert metric} is defined as 
\[
	h_{\mathbb{B}^2}(a, b) = \log \frac{| u- b||a-v |}{|u-a||b-v|} ,
\]
where $  u $, $  v $  are the intersection points of the line through $a$ and $b,$ 
and the unit circle $  \partial \mathbb{B}^2 $  ordered in such a way that
 $  \left| u - a \right|  < \left| u - b \right|  $. Another metric we study here
is the {\it visual angle metric} for $ a, b \in \mathbb{B}^2 $  defined by
\[
	v_{\mathbb{B}^2}(a, b) = \sup \left\lbrace  \alpha : \alpha = \measuredangle (a, z, b), z \in \partial   \mathbb{B}^2 \right\rbrace .
\]

It turns out that the visual angle metric provides a geometric interpretation
of the Hilbert metric. In fact, the following functional identity holds.

\begin{thm}  \label{thm11}
Let $a, b \in \mathbb{B}^2$ and $m = d( \left\lbrace 0\right\rbrace, L[a, b] )$
where $L[a,b]$ is the line through $a$ and $b.$ Then
the following functional identity holds
\begin{equation*}
\tan \frac{v_{\mathbb{B}^2}(a,b)}{2} 
= \frac{\sqrt{1+ m }}{\sqrt{1- m }} \th \frac{h_{\mathbb{B}^2}(a,b)}{4}.
\end{equation*}
\end{thm}

Above $\th$ stands for the hyperbolic tangent
function.
The proof of Theorem \ref{thm11} is based on the use of the hyperbolic metric in terms of
which both metrics can be expressed by the results in \cite{fkv} and \cite{rv}.

We apply this result to prove the following sharp distortion result for $K$-quasiregular
mappings \cite{lv}. These mappings form a very wide class of mappings in the plane:
$K$-quasiregular with the parameter value $K=1$ are analytic functions and injective
quasiregular mappings are quasiconformal mappings.

\begin{thm} \label{thm12}
	Let  $ a, b \in \mathbb{B}^2 ,   m = d(\{0\}, L[a, b]), $
and let	  $ f: \mathbb{B}^2 \to f(\mathbb{B}^2)=\mathbb{B}^2 $ be a $ K $-quasiregular mapping. Then
\begin{eqnarray} \label{12ineq}
\th \frac{h_{\mathbb{B}^2}(f(a), f(b))}{4} \leq D \left( \th \frac{h_{\mathbb{B}^2}(a, b)}{4} \right)^{1/K},
\end{eqnarray}
where
\begin{eqnarray*}
D = 2^{1 - 1/K} \left( \frac{1}{\sqrt{1 - m^2}} \right)^{1/K} \,.
\end{eqnarray*}
\end{thm}

The three main results of this paper are the above  Theorems \ref{thm11}, \ref{thm12} and
Theorem \ref{fujiEllip}. This last theorem studies circles in the Hilbert geometry.
It is shown that Hilbert circles are, in fact, Euclidean ellipses. Using this result
we can find the sharp radii for the hyperbolic incircles and circum circles of Hilbert
circles.


The paper is organized as follows: Section \ref{sec2} provides the  definition
of the hyperbolic metric and theoretical framework necessary for developing our results. In Section \ref{sec3}, we solve a geometric problem concerning hyperbolic distances between points
of intersection of lines joining four complex points on the unit circle $ \mathbb{S}^1 $. In Section \ref{sec4}, we  give the
results from  \cite{frv} and \cite{rv} expressing  the Hilbert and visual angle metrics in terms of the hyperbolic metric. In Section \ref{sec5}, we prove the
above two main results Theorems \ref{thm11} and \ref{thm12}. As far as we know,
the distortion result in Theorem \ref{thm12} is new also for analytic functions.
In Section \ref{sec8}, we  apply some computer algebra methods, which are similar to those of \cite{frv}, 
to prove Theorem \ref{fujiEllip}.

\section{Preliminary results} \label{sec2}
%
This section presents the foundational definitions, notations, and key results that will be utilized throughout the paper, focusing on complex geometry, M\"obius transformations, and hyperbolic metrics.

The $ n $-dimensional unit ball is expressed by $ \mathbb{B}^n $, while the unit sphere in $\mathbb{R}^n $ is expressed by $ \mathbb{S}^{n-1} $. 
For $ a \in \mathbb{R}^n \setminus \left\lbrace 0\right\rbrace  $, let \( a^* = \frac{a}{|a|^2} \).
The dot product of two points $ a, b \in \mathbb{R}^n $ is denoted by $ a \cdot b$. The cross-ratio of four points $ u, a, b, v \in \mathbb{R}^n $ is defined as
\begin{equation} \label{cr}
	\left| u, a, b, v\right|  = \frac{\left| u - b\right|  \left| a - v\right| }{\left| u - a\right|  \left| b - v\right| }
\end{equation}
and the Hilbert metric is defined as in (\ref{hilb}). For $ a \in \mathbb{C} $, $ \overline{a} $ is the complex conjugate of $ a $.

Assume that $ L[a, b] $ represents the line passing through points $ a $ and $ b $ ($ b \neq a $). For distinct points $ a, b, c, d \in \mathbb{C}$, if the lines $ L[a, b] $ and $ L[c, d] $ intersect at a single point $ w $, then

\begin{equation*}
w = \text{LIS}[a, b, c, d] = L[a, b] \cap L[c, d].
\end{equation*}
The coordinates of this intersection are given by (see e.g., \cite[Ex. 4.3(1), p. 57 and p. 373]{hkv})

\begin{equation}\label{2.1}
	w = \text{LIS}[a, b, c, d] = \frac{(\overline{a}b - a\overline{b})(c - d) - (a - b)(\overline{c}d - c\overline{d})}{(\overline{a} - \overline{b})(c - d) - (a - b)(\overline{c} - \overline{d})}.
\end{equation}

Let $ C[a, b, c]$ represent the unique circle through three distinct, non-collinear points $ a $, $ b $, and $ c $. The formula (\ref{2.1}) easily yields a
formula for the center of this circle, because this center is the point of intersection of the bisecting normals to
two sides of the triangle.

A M\"obius transformation is a mapping of the form
\begin{equation*}
	z \mapsto \frac{az + b}{cz + d}, 
	\ \ \ a, b, c, d, z \in \mathbb{C}, \ ad - bc \neq 0.
\end{equation*}
The most important features of M\"obius transformations are that they  preserve 
the cross-ratio and the angle magnitude, and, because of this, they map every Euclidean line or circle onto either a line or a circle. The special M\"obius transformation
\begin{equation} \label{mb}
	T_a(z) = \frac{z - a}{1 - \overline{a}z}, \ \ \ a \in \mathbb{B}^2 \setminus \left\lbrace 0\right\rbrace 
\end{equation}
maps the unit disk $ \mathbb{B}^2 $ onto itself with
\begin{equation*}
	T_a(a) = 0, \ \ \ T_a\left(\pm \frac{a}{\left| a\right| }\right) = \pm \frac{a}{\left| a\right| }.
\end{equation*}

\begin{nonsec}{\bf Hyperbolic geometry.}\label{hg}
We review some basic formulas and notation for hyperbolic geometry 
following \cite{b}.

The hyperbolic metrics of the unit disk ${\mathbb{B}^2}$ and the upper half plane  ${\mathbb{H}^2}$ are given, respectively, by
\begin{equation}\label{rhoB}
\sh \frac{\rho_{\mathbb{B}^2}(a,b)}{2}=
\frac{\left| a-b\right| }{\sqrt{(1-\left| a\right| ^2)(1-\left| b\right| ^2)}} ,\ \ \ a,b\in \mathbb{B}^2,
\end{equation}
and
\begin{equation}\label{rhoH}
{\rm ch}\rho_{\mathbb{H}^2}(a,b)=1+ \frac{\left| a-b\right| ^2}{2 {\rm Im}(a) {\rm Im}(b)},\ \ \ a,b\in \mathbb{H}^2\,.
\end{equation}
Both metrics are M\"obius invariant: if $G, D \in \left\lbrace  \mathbb{B}^2, \mathbb{H}^2\right\rbrace $
and $f:G \to D= f(G)$ is a M\"obius transformation, then  $\rho_G(a,b)= \rho_D(f(a),f(b))$ holds for all $a,b \in G$.
\end{nonsec}

We shall use the fact that a hyperbolic disk $ B_\rho(x,M)$ with the center $x \in {\mathbb{B}}^2$ and the radius $M>0$ is a Euclidean disk with the following
center and radius \cite[p. 56, (4.20)]{hkv}
\begin{equation}\label{hkv420}
 \begin{cases}
        B_\rho(x,M)=B^2(y,r)\;,&\\
  \noalign{\vskip5pt}
      {\displaystyle y=\frac{x(1-t^2)}{ 1-|x|^2t^2}\;,\;\;
        r=\frac{(1-|x|^2)t}{1-|x|^2t^2}\;,\;\;t={\rm th} ( M/2)\;.}&
\end{cases}
\end{equation}
Above the symbols ${\rm sh},$ ${\rm ch},$ and ${\rm th}$ stand for the hyperbolic
sine, cosine, and tangent functions. Their inverses are  ${\rm arsh},$ ${\rm arch},$ and ${\rm arth}\,.$



\bigskip
\section{An observation about hyperbolic metric}\label{sec3}
We solve here the following claim which was formulated as Problem 5.10 in \cite{frv}.

\begin{nonsec}{\bf  Claim.}\label{prob1} Let $ a, b, c, d $ be four complex points on the unit circle $ \mathbb{S}^1 $ in this order so that $ L[a, b] $ and $ L[c, d] $ are not parallel. Let $ h $ be an arbitrary point on the Euclidean segment $ [b, c] $, and fix then  $ g = \mathrm{LIS}[a, b, c, d] $, $ j = \mathrm{LIS}[g, h, a, c] $, $ k = \mathrm{LIS}[g, h, b, d] $, and  $ l = \mathrm{LIS}[g, h, a, d] $.  Note that the special case $ j = k $ is possible. Now,  
$ \rho_{\mathbb{B}^2}(h, j) = \rho_{\mathbb{B}^2}(k, l) $ (See Figure \ref{fig1}).
\end{nonsec}

\bigskip

\begin{thm}\label{fujThm} The claim formulated above
holds true.
\end{thm}

\begin{proof}

%

\begin{figure}[H] 
\includegraphics[width=0.6\linewidth]{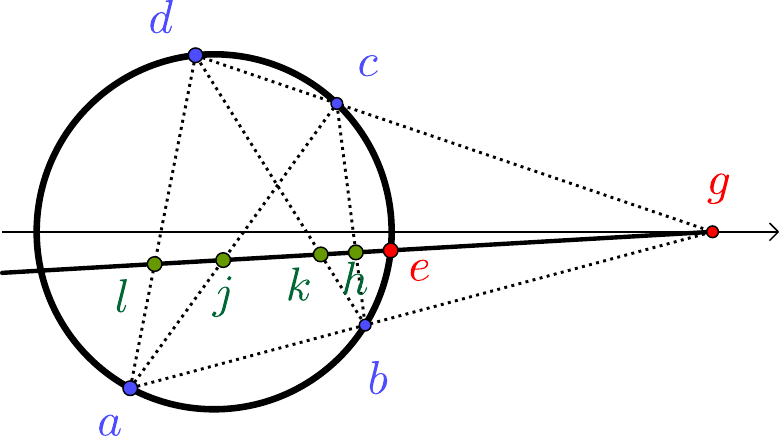}
		\caption{$\rho_{\mathbb{B}^2}(h,k) = \rho_{\mathbb{B}^2}(l,j)$}
		\label{fig1}
\end{figure}	
	
Let $ g $ be a positive real number, and let $ a$ and $d $ be two complex numbers 
on $ \partial\mathbb{B}^2 $.
Let $ L[a,g] $ be the line passing through $ a $  and $ g $:
	\begin{equation*}
		 L[a,g]\ :\ 
		(g-\overline{a})z+(a-g)\overline{z}-ag+\overline{a}g=0.
	\end{equation*}
	The intersection points of $ L[g,a] $ and $ \partial\mathbb{B}^2 $
	are given by the solution of the following equation:
	\begin{equation*}
		 \Big(g-\frac{1}{a}\Big)z+(a-g)\frac{1}z-ag+\frac{1}{a}g
		=0.
	\end{equation*}
	Rearranging terms, we have:
	
	\begin{equation*}
		(z - a) \left( (ag - 1)z - a + g \right) = 0.
	\end{equation*}
Let $ b $ be an intersection point different from $ a $, i.e.,
	\begin{equation*} 
		b=\frac{g-a}{1-ga}.
			\end{equation*}
Similarly, let 	\begin{equation*} 
	c=\frac{g-d}{1-gd} 	
	\end{equation*}
with respect to $ d $. Next, take a point $ e $ on the arc $ \overset{\frown}{bc}$. Let
\begin{equation*}
	h = \text{LIS}\left[ b, c, g, e \right] , \ \ \ k = \text{LIS}\left[ a, c, g, e\right] , \ \ \ j = \text{LIS}\left[ b, d, g, e\right] , \ \ \ l = \text{LIS}\left[ a, d, g, e\right] .
\end{equation*}
Then, we have the following (by using Risa/Asir) \cite{frv,n}:
$$
\frac{h-j}{1-\overline{j}h} 
=\frac{e(g-e)((a^2+1)g-2a)}
{((e^2-2ae-1)g^2+((a^2+1)e+2a)g-e^2-a^2}
=\frac{k-l}{1-\overline{l}k} 
$$
and
$$
\frac{h-k}{1-\overline{k}h}
= \frac{e(g-e)((d^2+1)g-2d)}
{(e^2-2de-1)g^2+((d^2+1)e+2d)g-e^2-d^2}
=\frac{j-l}{1-\overline{l}j}.
$$
From \eqref{rhoB} it follows that
\begin{eqnarray}\label{xx}
	\th \frac{\rho(z,w)}{2} = \frac{\left| z-w \right| }{\left|1-z\bar{w} \right| }.
\end{eqnarray}
Using (\ref{xx}) we have
$$
\rho_{\mathbb{B}^2}(h,j)=2\mbox{arth}\,\Big|\frac{h-j}{1-\overline{j}h}\Big|
=2\mbox{arth}\,\Big|\frac{k-l}{1-\overline{l}k}\Big|
=\rho_{\mathbb{B}^2}(k,l),
$$
and
$$
\rho_{\mathbb{B}^2}(h,k)=2\mbox{arth}\,\Big|\frac{h-k}{1-\overline{k}h}\Big|
=2\mbox{arth}\,\Big|\frac{j-l}{1-\overline{l}j}\Big|
=\rho_{\mathbb{B}^2}(j,l).
$$
The proof is completed.

\end{proof}

\begin{rem} 
We can see geometrically that the equality holds even if the points $ k $ and $ j $ are swapped with each other.

	As the point $ e $ moves on the arc $  \overset{\frown}{bc} $, the positions of $ j $ and $ k $ may be swapped with each other. This is because $ j $ and $ k $ are defined as the intersections of two lines (compare Figures \ref{fig1} and \ref{fig2}).
\end{rem}
	\begin{figure}[H] 
	\centering
	\begin{minipage}{0.48\textwidth}
		\centering
\includegraphics[width=\linewidth]{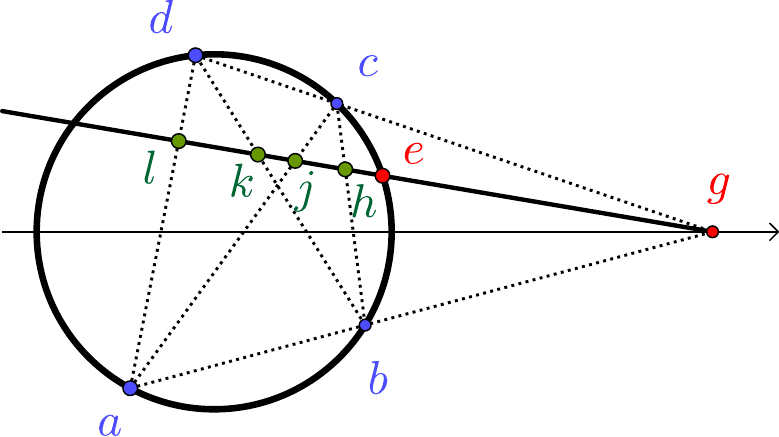}
		\caption{}
	\label{fig2}
	\end{minipage}
	\hfill
	\begin{minipage}{0.48\textwidth}
		\centering
		\includegraphics[width=\linewidth]{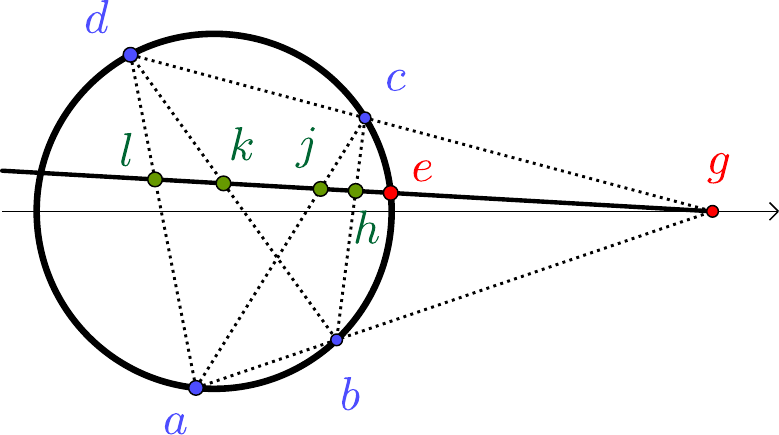}
		\caption{}
	\label{fig3}
	\end{minipage}
\end{figure}

In fact, by setting $ e $ to $ \bar{e} $, $ d $ to $ \bar{a} $, and $ a $ to $ \bar{d} $, we can obtain a figure that is symmetric with respect to the real axis with the original figure (Figure \ref{fig1}). Then, the positions of $ j $ and $ k $ are swapped with each other (compare Figures \ref{fig1} and \ref{fig3}).

%




\section{The Hilbert metric and the visual angle metric of the unit disk} \label{sec4}

For all distinct points $ a $ and $ b $ in a bounded convex domain $ G \subset \mathbb{R}^n $, the Hilbert metric is defined as \cite[Thm 2.1, p. 157]{b2}
	\begin{equation} \label{hilb}
	h_G(a, b) = \log \left| u, a, b, v \right|,
\end{equation}
where $  u $, $  v $  are the intersection points of the line $  L[a, b] $  and the domain boundary $  \partial G $  ordered in such a way that
 $  \left| u - a \right|  < \left| u - b \right|  $. See the definition of the cross-ratio from (\ref{cr}). If $  a = b $, we set $  h_G(a, b)=0$. Hilbert \cite{hilb} introduced this metric $  h_G $  as an extension of the Klein metric for any bounded convex domain $ G $. 

Unlike the hyperbolic metric $ \rho_{\mathbb{B}^2} $, the Hilbert metric is not invariant under the M\"obius automorphisms of $ \mathbb{B}^2 $ as indicated by the following theorem.

\begin{thm} (See \cite[Thm 1.2]{rv}) \label{4.2}
For all $ a, b \in \mathbb{B}^2 $, the following functional identity holds 
between the	Hilbert metric and the hyperbolic metric:
\begin{equation*}
		\sh \left( \frac{h_{\mathbb{B}^2}(a, b)}{2} \right) = 
		\sqrt{1 - m^2} \sh \left( \frac{\rho_{\mathbb{B}^2}(a, b)}{2} \right),
\end{equation*}
where $ m $ is the Euclidean distance from the origin to the line $ L[a, b] $.
\end{thm}
	
\begin{thm} Under the conditions given in Claim \ref{prob1}, 
the following equalities hold for Hilbert metric:
		\begin{enumerate}
			\item $h_{\mathbb{B}^2}(h, j) = h_{\mathbb{B}^2}(k, l)$,
			\item $h_{\mathbb{B}^2}(h, k) =h_{\mathbb{B}^2}(j, l)$.
		\end{enumerate}
\end{thm}
	
\begin{proof}
The proof follows directly by applying Theorems \ref{4.2}  and \ref{fujThm}.
\end{proof}

Let $G $ be a proper subdomain of $ \mathbb{R}^n $ such that $ \partial G $ is not a proper subset of a line. The visual angle metric for $ a, b \in G $ is defined by
\begin{equation} \label{vam}
	v_G(a, b) = \sup \left\lbrace  \alpha : \alpha = \measuredangle (a, z, b), z \in \partial G \right\rbrace .
\end{equation}
That is, the visual angle metric measures the maximal visual angle 
$\measuredangle(a, z, b)$ between the points $a$ and $b$ at the point $z$ on the boundary $ \partial G .$ 

\begin{thm} (See \cite[Thm 1.3]{fkv}) \label{5.2}
For $ a, b \in \mathbb{B}^2 $, we have
\begin{equation}
		\tan \frac{v_{\mathbb{B}^2}(a, b)}{2} =
		\frac{(1 +  m  )u}{1 + \sqrt{1 + (1 -   m^2)u^2}},
		\ \ \ u = \sh \frac{\rho_{ \mathbb{B}^2}(a, b)}{2},
\end{equation}
where $ m = |\frac{ab - \overline{a}\overline{b}}{2(a - b)}| $ is the absolute
value of the midpoint of the chord of the unit disk containing the two points 
$a $ and $ b,$ and hence $  m  = d( \left\lbrace 0\right\rbrace, L[a, b] ) $.
\end{thm}

\begin{thm} Under the given conditions in Claim \ref{prob1}, 
the following equalities hold for the visual angle metric:
\begin{enumerate}
		\item  $v_{\mathbb{B}^{2}}(h, j) = v_{\mathbb{B}^{2}} (k, l)$,
		\item $v_{\mathbb{B}^{2}}(h, k) = v_{\mathbb{B}^{2}}(j, l)$.
\end{enumerate}
\end{thm}

\begin{proof}
The proof follows directly by applying Theorems \ref{5.2}  and
\ref{fujThm}.
\end{proof}

The proof of Theorem \ref{5.2} makes use of the inversion $\tau : {\mathbb{B}^2} \to {\mathbb{B}^2} = \tau({\mathbb{B}^2}) $ mapping the chord $L[a,b] \cap  {\mathbb{B}^2}$ onto itself with $|\tau(a)| = |\tau(b)|.$ Under this transformation we also have $v_{{\mathbb{B}^2}}(a,b) = v_{{\mathbb{B}^2}}(\tau(a),\tau(b))$ and,
by symmetry, $ v_{{\mathbb{B}^2}}(\tau(a),\tau(b))$ can be easily computed.
	
\begin{figure}[H] 
	\centering
	\includegraphics[width=0.7\linewidth]{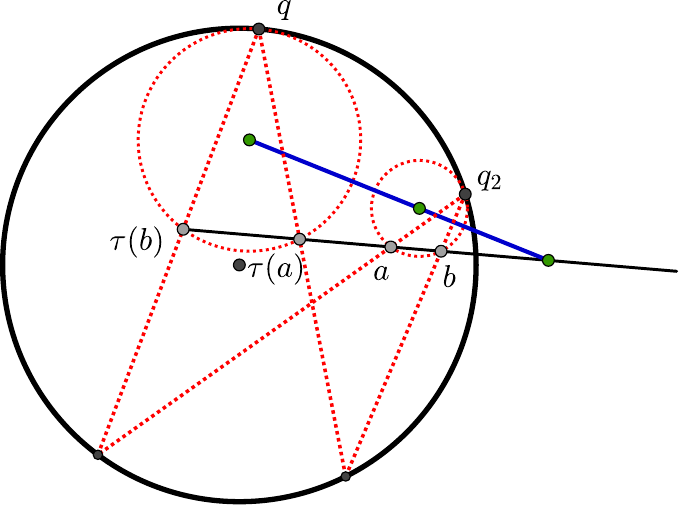}
	\caption{ Under the inversion $\tau,$
	$v_{{\mathbb{B}^2}}(a,b) = v_{{\mathbb{B}^2}}(\tau(a),\tau(b))$
and $\rho_{{\mathbb{B}^2}}(a,b) = \rho_{{\mathbb{B}^2}}(\tau(a),\tau(b)).$	
Observe that $\measuredangle(a, q_2, b)= \measuredangle(\tau(a),q,\tau(b))$
$ =v_{{\mathbb{B}^2}}(a,b)\, $
and $\tau(q_2)=q.$
	}
	\label{fig4}
\end{figure}




\bigskip

\begin{lem}\label{lem:H}
For $ a,b\in\mathbb{B}^2 $, the Hilbert distance $ h_{\mathbb{B}^2}(a,b) $
is given by $ \log H $, where
$$
  H=\dfrac{{\rm Re\,}\big((1-a\overline{b})^2\big)+|a-b|^2
           +2\,{\rm Re\,}(1-a\overline{b})
           \sqrt{|a-b|^2-\big({\rm Im\,}(a\overline{b})\big)^2}}
          {(1-|a|^2)(1-|b|^2)}.
$$
\end{lem}
\bigskip

\begin{proof}
The equation of the line $ L[a,b] $ is given by
$$
   (\overline{a}-\overline{b})z-(a-b)\overline{z}=\overline{a}b-a\overline{b}.
$$
Let $ u,v $ be the intersection points of the line $ L[a,b] $ and 
the unit circle $ \mathbb{S}^1 $.
So, the points $ u,v $  are the solutions to the quadratic equation
$$
   (\overline{a}-\overline{b})z^2-(\overline{a}b-a\overline{b})z-(a-b)=0.
$$
From the relation between the roots and the coefficients of a quadratic equation, we have
\begin{equation}\label{eq:uv}
    u+v=\frac{\overline{a}b-a\overline{b}}{\overline{a}-\overline{b}},
    \qquad \textrm{and} \qquad
    uv=-\frac{a-b}{\overline{a}-\overline{b}}.
\end{equation}
Then, the exponential $H$ of the Hilbert distance is given by
$$
  H= \max\bigg\{\frac{|u-b||v-a|}{|u-a||v-b|},\ 
             \frac{|u-a||v-b|}{|u-b||v-a|}\bigg\}.
$$
The two elements above are reciprocals of each other. 
Moreover, we remark that the both values $ \frac{(u-b)(v-a)}{(u-a)(v-b)} $
and $ \frac{(u-a)(v-b)}{(u-b)(v-a)} $ are real because $ a,b,u,v $
are collinear. So, 
\begin{equation}\label{eq:hilab}
   H=\max\bigg\{\frac{(u-b)(v-a)}{(u-a)(v-b)},\ 
             \frac{(u-a)(v-b)}{(u-b)(v-a)}\bigg\}.
\end{equation}
Eliminating $ u,v $ from \eqref{eq:hilab} and  \eqref{eq:uv} gives us
\begin{equation}\label{eq:eqhil}
   (1-a\overline{a})(1-b\overline{b})H^2
   -\big((1-a\overline{b})^2+(1-\overline{a}b)^2
        +2(a-b)(\overline{a}-\overline{b})\big)H
       +(1-a\overline{a})(1-b\overline{b})=0.
\end{equation}
Note that the above equation has two solutions that are 
reciprocals of each other.
Therefore, the larger one gives $ H $. In fact,
\begin{align*}
  H & = \frac{(1-a\overline{b})^2+(1-\overline{a}b)^2+2|a-b|^2+
           |2-a\overline{b}-\overline{a}b|
           \sqrt{4|a-b|^2-|a\overline{b}-\overline{a}b|^2}}
            {2(1-|a|^2)(1-|b|^2)} \\
    & = \frac{\textrm{Re\,}\big((1-a\overline{b})^2\big)+|a-b|^2
           +2\textrm{Re\,}(1-\overline{a}b)
            \sqrt{|a-b|^2-\big(\textrm{Im}(a\overline{b})\big)^2}}
             {(1-|a|^2)(1-|b|^2)}.
\end{align*}
Hence, the assertion is obtained.
\end{proof}

\section{A functional identity for the Hilbert metric} \label{sec5}
In this section, we apply Theorems \ref{4.2} and \ref{5.2} to prove  Theorem \ref{thm11} and apply it to study distortion under quasiregular mappings.

\begin{nonsec}{\bf The proof of Theorem \ref{thm11}. } 
By Theorem \ref{4.2}, we have
\begin{equation}
	\label{eq:6.2}
	\sh \frac{\rho_{\mathbb{B}^2}(a,b)}{2} 
	= \frac{1}{\sqrt{1-m^2}} \sh \frac{h_{\mathbb{B}^2}(a,b)}{2}.
\end{equation}
Theorem \ref{5.2} yields
\begin{equation}
	\label{eq:6.3}
	\tan \frac{v_{\mathbb{B}^2}(a,b)}{2} 
	= \frac{(1+m) \sh \frac{\rho_{\mathbb{B}^2}(a,b)}{2}}
	{1 + \sqrt{1+(1-m^2)\sh^2\frac{\rho_{\mathbb{B}^2}(a,b)}{2}}}.
\end{equation}
Substitution of \eqref{eq:6.2} into \eqref{eq:6.3} yields
\begin{eqnarray*}
	\begin{array}{ll}
		\displaystyle \tan \frac{v}{2} & \displaystyle = \sqrt{\frac{1+m}{1-m}} \frac{ \sh \frac{h}{2}}
		{ 1 + \sqrt{1+\sh^2\frac{h}{2}} } =\sqrt{\frac{1+m}{1-m}} \frac{ \sh \frac{h}{2}}
		{ 1 + \ch \frac{h}{2} } \\
		&  \\
		& \displaystyle = \sqrt{\frac{1+m}{1-m}}\, \th \frac{h}{4},
	\end{array}
\end{eqnarray*}
where $v = v_{\mathbb{B}^2}(a,b)$ and $h = h_{\mathbb{B}^2}(a,b)$. \hfill $\square$
\end{nonsec}

We next apply the functional identity of Theorem \ref{thm11} to study distortion under quasiregular mappings. For these mappings we refer the reader to \cite{hkv,lv}.

\begin{nonsec}{\bf Proof of Theorem \ref{thm12}.}
By Theorem 1.5 in \cite{fkv}
\begin{eqnarray} \label{5.5}
\tan \frac{v_{\mathbb{B}^2}(f(a), f(b))}{2} \leq 2^{1-1/K}\, c \, \left( \tan \frac{v_{\mathbb{B}^2}(a, b)}{2} \right)^{1/K},
\end{eqnarray}
where
\begin{eqnarray*}
c = \sqrt{\frac{1 + m_1}{1 - m_1}} \cdot \frac{1}{(1 + m)^{1/K}}
\end{eqnarray*}
and $ m_1 = d(\{0\}, L[f(a), f(b)]) $.
Applying Theorem \ref{thm11}, we can write (\ref{5.5}) as follows:
\[\displaystyle \th \frac{h_{\mathbb{B}^2}(f(a), f(b))}{4} \displaystyle  \leq D\,
\left( \th \frac{h_{\mathbb{B}^2}(a, b)}{4} \right)^{1/K} \,\]
\[
\displaystyle  D = \sqrt{\frac{1 - m_1}{1 + m_1}} \sqrt{\frac{1 + m_1}{1 - m_1}}\frac{2^{1 - 1/K}}{\left( 1 + m\right)^{ 1/K}}\left( \sqrt{\frac{1 + m}{1 - m}}\right) ^{1/K} =\displaystyle  2^{1 - 1/K} \left( \frac{1}{\sqrt{1 - m^2}} \right)^{1/K}\,.\]
\hfill $\square$
\end{nonsec}

\medskip

\begin{rem}{\label{sharp}}
Theorem \ref {thm12} is sharp. We outline here how the sharpness can be proven
for a M\"obius transformation $T_w,$ \eqref{mb},  when $K=1.$ To this effect, fix $w = 0.9$ and
choose $t \in (0,1)$ and denote $a= T_w^{-1}(it), b= T_w^{-1}(-it)\,. $ Then
the parameter $m= |{\rm Re}(a)|=  |{\rm Re}(b)|$ and one can show by computer tests
that the quotient of the two sides of the inequality \eqref{12ineq} tends to $1$ when
$t \to 0.$

\end{rem}

\medskip

Theorem \ref {thm12} is apparently new also for the case of analytic functions. In fact,
we are not familiar with any distortion results on analytic functions, expressed
in terms of the Hilbert metric.

\medskip
\begin{cor}
	Let $ f: \mathbb{B}^2 \to f(\mathbb{B}^2) = \mathbb{B}^2 $ be a $K$-quasiregular mapping and $ a, b \in \mathbb{B}^2$, and let $ 0 \in L[a, b] $.
\begin{enumerate}
	\item Then
\begin{eqnarray*}
	\th \frac{h_{\mathbb{B}^2}(f(a), f(b))}{4} \leq 
	2^{1 - 1/K} \left( \th \frac{\rho_{\mathbb{B}^2}(a, b)}{4} \right)^{1/K}.
\end{eqnarray*}
	\item If both $ 0 \in L[a, b] $ and $ 0 \in L[f(a), f(b)] $ hold, then
\begin{eqnarray*}
	\th \frac{\rho_{\mathbb{B}^2}(f(a), f(b))}{4} \leq 
	2^{1 - 1/K} \left( \th \frac{\rho_{\mathbb{B}^2}(a, b)}{4} \right)^{1/K}.
\end{eqnarray*}
\end{enumerate}
\end{cor}
\begin{proof}
\begin{enumerate}
	\item Because $ 0 \in L[a, b] $, we have $m=0$ and  it follows from Theorem \ref{4.2} that 
\begin{eqnarray*}
	h_{\mathbb{B}^2}(a, b) = \rho_{\mathbb{B}^2}(a, b),
\end{eqnarray*}
and hence the constant $D $ in Theorem \ref{thm12} equals $ 2^{1 - 1/K} $. 
	
	\item Under the additional assumption $ 0 \in L[f(a), f(b)] $, Theorem \ref{4.2} also yields
\begin{eqnarray*}
	h_{\mathbb{B}^2}(f(a), f(b)) = \rho_{\mathbb{B}^2}(f(a), f(b)).
\end{eqnarray*}
Substituting this into the inequality of the first part immediately gives the assertion. We note that a related result for the hyperbolic metric can be found in \cite[Theorem 16.2]{hkv}. 
\end{enumerate}
\end{proof}
\small

\normalsize



\section{Hilbert circles} \label{sec8}

We consider here Hilbert circles in  the unit disk and, in particular, we compare the Hilbert circles to Euclidean circles. We apply  Gr\"obner bases from 
computer algebra  in the same way as in \cite{frv}  to prove that Hilbert circles are Euclidean
ellipses.

We first study the defining equation  of the Hilbert circle $ \partial B_h(z_0,t) .$

\begin{thm} \label{fujiEllip}
For $ z_0\in\mathbb{B}^2 $, the boundary $ \partial B_h(z_0,t) $ of
the Hilbert disk forms the ellipse defined by the equation 
\begin{align} \notag
   & \overline{z_0}^2rz^2
     -\big((r^2+1)|z_0|^2-(r+1)^2\big)z\overline{z}+z_0^2r\overline{z}^2 \\
       \label{eq:hil-equi}
   & \qquad
     -4\overline{z_0}rz-4z_0r\overline{z}+(r+1)^2|z_0|^2-(r-1)^2=0,
\end{align}
where $ r=e^t $.
\end{thm}

\begin{rem} \label{semiaxes}
If $ z_0\in\mathbb{R} $, the equation \eqref{eq:hil-equi} can be 
written as
\begin{equation}\label{eq:equi-real1}
  \big((r+1)^2-(r-1)^2z_0^2\big)x^2-8z_0rx 
   +(1-z_0^2)(r+1)^2y^2+(r+1)^2z_0^2-(r-1)^2=0,
\end{equation}
where $ z=x+iy $.
The above equation is also expressed by
\begin{equation}\label{eq:equi-real2}
    \frac{1}{(r^2-1)^2(1-z_0^2)^2}
        \Big(\big((r+1)^2-(r-1)^2z_0^2\big)x-4rz_0\Big)^2
    +\frac{(r+1)^2-(r-1)^2z_0^2}{(r-1)^2(1-z_0^2)}y^2=1.
\end{equation}
So, the semi-minor and semi-major axes are
$$
   \frac{(r^2-1)(1-z_0^2)}{(r+1)^2-(r-1)^2z_0^2} 
   \quad\textrm{and}\quad
   \frac{(r-1)\sqrt{1-z_0^2}}{\sqrt{(r+1)^2-(r-1)^2z_0^2}},
$$
respectively.
\end{rem}
\bigskip

\begin{proof}
  From Lemma \ref{lem:H}, for $ a,b\in\mathbb{B}^2 $, if we set
  $ h_{\mathbb{B}^2}(a,b)=\log H $, then $ H $  satisfies \eqref{eq:eqhil}.
  Using the settings here, substituting $ a=z_0,\ b=z $ and $ H=e^t=r $
  into  \eqref{eq:eqhil} gives
  $$
   (1-z_0\overline{z_0})(1-z\overline{z})r^2
   -\big((1-z_0\overline{z})^2+(1-\overline{z_0}z)^2
        +2(z_0-z)(\overline{z_0}-\overline{z})\big)r
       +(1-z_0\overline{z_0})(1-z\overline{z})=0.
  $$
  Rearranging the above equation, we obtain \eqref{eq:hil-equi}.
\end{proof}

\bigskip


\begin{rem}
In fact, from 
$$
    ((1+r)^2-(1+r^2)|z_0|^2)^2-4|z_0|^4r^2
    =(r+1)^2(1-|z_0|^2)((r+1)^2-(r^2+1)|z_0|^2)>0,
$$
we can check that \eqref{eq:hil-equi}
gives an equation of an ellipse.
\end{rem}

%

\begin{thm} \label{hilhyp0115}
For $ 0<z_0<1 $ and $ t >0$, the following hold.
\begin{enumerate}
  \item \label{item:1}
         The largest number $ s $ satisfying
         $$ B_{\rho}(z_0,s) \subset B_h(z_0,t) $$
       is $ s=t $, see  Figure 5.
  \item \label{item:2}
       The smallest number $s$ satisfying 
        $$    B_h(z_0,t) \subset B_{\rho}(z_0,s) $$
       is given by $s=2\operatorname{arth} R$, where
        $$
           R=\frac{r-1}{\sqrt{(r+1)^2-4rz_0^2}}, \ \ r= e^t.
        $$
\end{enumerate}
\end{thm}

\begin{proof}
By \eqref{hkv420} the hyperbolic circle $ \partial B_{\rho}(z_0,s) $ coincides with
the Euclidean circle with the radius $ (1-z_0^2)R/(1-z_0^2R^2) $
and the center at $ z_0(1-R^2)/(1-z_0^2R^2) $, where $ R={\rm th}\frac{s}2 .$

Here, we consider the Hilbert circle $ \partial B_h(z_0,t) $ and
the hyperbolic circle $ \partial B_{\rho}(z_0,s) $.
The equations of these curves are given by,
\begin{equation}\label{eq:hil}
  Hil(z)=rz_0^2z^2+((r+1)^2-(r^2+1)z_0^2)z\overline{z}
     +rz_0^2\overline{z}^2-4rz_0z-4rz_0\overline{z}+(r+1)^2z_0^2-(r-1)^2=0,
\end{equation}
and
\begin{equation}\label{eq:hyp}
  Hyp(z)=(R^2z_0^2-1)z\overline{z}-(R^2-1)z_0z-(R^2-1)z_0\overline{z}
     -z_0^2+R^2=0,
\end{equation}
where $ r=e^t $ and $ R={\rm th}\frac{s}2 $.

We need to find the conditions that 
the Hilbert circle and the hyperbolic circle are tangent to each other.

First, eliminating $ \overline{z} $ from 
$ Hil=0 $ and $ Hyp=0 $ by computing the resultant, we obtain
$$
   H_1(z)=\textrm{resultant}_{\overline{z}}(Hil,Hyp)=0.
$$
In fact, the coefficients of the polynomial $ H_1(z)=c_4z^4+c_3z^3+c_2z^2+c_1z+c_0$ are
\begin{align*}
 c_4 & = rz_0^2(R^2z_0^2-1)^2,\\
 c_3 &=-z_0(R^2z_0^2-1)\Big(\big((R^2-1)(r+1)^2+4rR^2\big)z_0^2
       -\big(R^2(r+1)^2-(r-1)^2\big)\Big),\\
 c_2 &= R^2\big((R^2-1)(r^2+1)+2R^2r\big)z_0^6
        +(R^2-1)\big((R^2-2)(r^2+1)+2(4R^2-1)r\big)z_0^4 \\
     & \qquad
        -\big((R^2-1)(2R^2-1)(r^2+1)+2(2R^4+2R^2-1)r\big)z_0^2
        +\big((r+1)^2R^2-(r-1)^2\big),\\
 c_1 &=-\big((R^2-1)(2R^2-1)(r^2+1)+2(2R^4-R^2+1)r\big)z_0^5\\
     & \qquad
      +\big((R^4-1)(r^2+1)-2(R^4-4R^2-1)r\big)z_0^3
      +(R^2-2)\big(R^2(r+1)^2-(r-1)^2\big)z_0,\\
 c_0 &= rz_0^6+\big((R^2-1)^2(r^2+1)+2(R^4-R^2-1)r\big)z_0^4
        -\big((R^2-1)^2(r^2+1)+(R^4-2)r\big)z_0^2.
\end{align*}
Next, to find the condition that the equation $ H_1=0 $ 
has multiple solutions, we compute the resultant again,
$$
  H_2(z)=\textrm{resultant}_{{z}}(H_1,\frac{\partial}{\partial z}H_1)=0.
$$
In fact,
\begin{align*}
  H_2(z) &=R^4(R-1)^2(R+1)^2r(r+1)^4
         z_0^6(z_0-1)^{12}(z_0+1)^{12}
         (Rz_0-1)^4(Rz_0+1)^4 \\
        & \quad
         \times
         \big((R+1)r+R-1\big)^2
         \big((R-1)r+R+1\big)^2
        \Big(\big(4rz_0^2-(r+1)^2\big)R^2+(r-1)^2\Big)^2=0.
\end{align*}
Since $ 0<R<1 $, $ r>1 $ and $ 0<z_0<1 $, 
the last two factors are significant.
\begin{itemize}
  \item In the case that $ (R-1)r+R+1=0 $, we have
      $$
         R=\frac{1-r}{1+r}=\frac{1-e^t}{1+e^t}={\rm th} \frac{t}2.
      $$
      Since $ R={\rm th} \frac{s}{2} $, we have $ s=t $.
      In this case, $ \partial B_{\rho}(z_0,t) $ is
      inscribed in $ \partial B_h(z_0,t) $ and
      we have the assertion of item (\ref{item:1}).  

  \item In the case of $\big(4rz_0^2-(r+1)^2\big)R^2+(r-1)^2=0$, we have
      $$
          R=\frac{r-1}{\sqrt{(r+1)^2-4rz_o^2}},\quad \textrm{where}\quad
                        R={\rm th} \frac{s}2.
      $$
      In this case, $ \partial B_{\rho}(z_0,s) $ is
      circumscribed about $ \partial B_h(z_0,t) $ and 
      we have the assertion of item (\ref{item:2}).  
\end{itemize}
\end{proof}

 It follows from Theorem \ref{4.2} that for $0 < r < s < 1$, we have
\begin{equation}\label{71}
	\rho_{\mathbb{B}^2}(r, s) = h_{\mathbb{B}^2}(r, s)
\end{equation}
and we see that for $z_0 \in (0, 1)$
\begin{equation}\label{72}
	B_\rho(z_0, t) \subset B_h(z_0, t).
\end{equation}
This inclusion relation is illustrated in Figure \ref{fig:fig4}.
	\begin{figure}[htbp]
	\centering
	\includegraphics[width=0.35\linewidth]{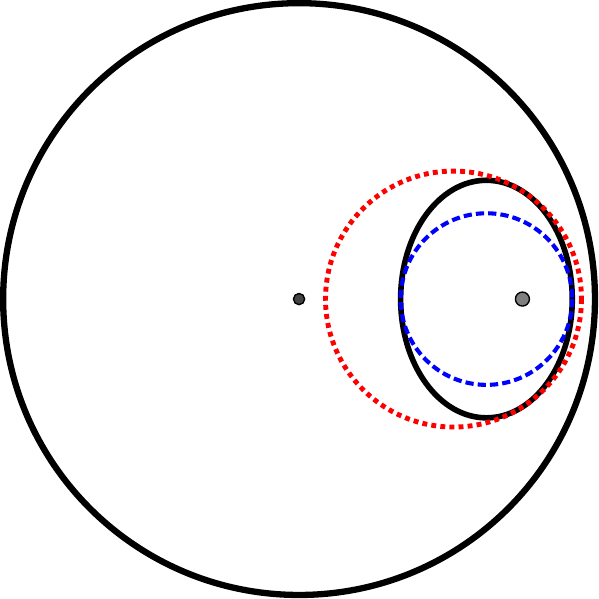}
	\caption{A Hilbert circle and a maximal inscribed hyperbolic circle and
	a minimal circumscribed hyperbolic circle with the same
	center, cf. Theorem \ref{hilhyp0115} for details.}
		\label{fig:fig4}
\end{figure}

In the case of a Euclidean circle, we have for $z_0 \in (0, 1)$, $s \in (0, 1 - |z_0|)$
\begin{equation} \label{73}
	B^2(z_0, s) \subset B_h(z_0, \rho_{\mathbb{B}^2}(z_0, z_0 + s))
\end{equation}
and this is again sharp.

\begin{figure}[htbp] 
	\centering
	
\includegraphics[width=0.4\linewidth]{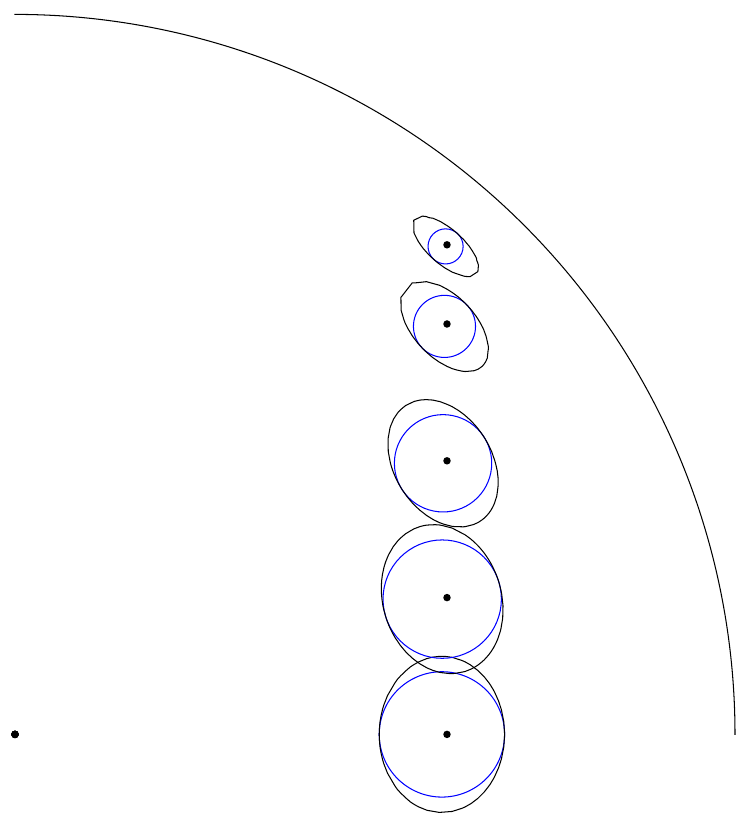}
		
		\label{fig:fig12}

 \caption{ Some Hilbert circles with centers $c_j, j=1,2,..$ on the segment $\{z  : {\rm Re}\{z\}=0.6, \quad 0 \le {\rm Im}\{z\}<\sqrt{1- 0.6^2} \} $ with
		radii $h_{\mathbb{B}^2}(0.6, 0.68)$ and the corresponding maximal inscribed hyperbolic circles
		with the same centers  $c_j, j=1,2,..$ and the same radii  $h_{\mathbb{B}^2}(0.6, 0.68)$.
		Observe that the Hilbert circles became more flat when the centers approach
		the boundary of the unit circle.}
\end{figure}

\begin{nonsec}{\bf Hilbert midpoint.} For $a,b \in \mathbb{B}^2,$ the 
{\it Hilbert midpoint} is a point of the form  $c= (1-t)a+ t b, \quad 0 < t <1,$ 
with $h_{ \mathbb{B}^2}(a,c) = h_{ \mathbb{B}^2}(c,b).$
\end{nonsec}

\begin{thm}
 For $a,b \in \mathbb{B}^2, \left| a\right| \neq \left| b\right| ,$ the 
{\it Hilbert midpoint} is
\[
c= (1-t)a+ t b, \quad \displaystyle t= \frac{-(1-|a|^2) + \sqrt{(1-|a|^2)(1-|b|^2)}}{|a|^2- |b|^2} \,.
\]
Moreover, $\rho_{\mathbb{B}^2}(a,c)=\rho_{\mathbb{B}^2}(c,b) \,. $
\end{thm}

\begin{proof}
Let $u$ and $v$ be the points of intersection of the line $L[a,b]$ with the unit 
circle ordered in such a way that $|u-a|<|u-b|\,.$ By the definition of 
the midpoint  $h_{ \mathbb{B}^2}(a,c) = h_{ \mathbb{B}^2}(c,b)$ or, in other words,
\begin{equation*} 
\displaystyle
\frac{(c-u)(v-a)}{(a-u)(c-v)} = \frac{(b-u)(v-c)}{(c-u)(b-v)} \,.
\end{equation*}
Then
\begin{equation} \label{my1}
(c-u)^2(v-a)(v-b)= (v-c)^2(a-u)(b-u) \,.
\end{equation}
Hence $u$ and $v$ are the solution of
\[
(\overline{a} - \overline{b})z^2 - (\overline{a} b - a \overline{b}) z - (a-b) = 0.
\]
Therefore,
\begin{equation} \label{my3}
u+v = \overline{a}b - a \overline{b}\,,\quad {\rm and} \,\, \quad uv = - \frac{a-b}{\overline{a}-\overline{b}} \,.
\end{equation}
Eliminating $u, v$ from \eqref{my1}, we have
the desired value of $t\,.$

The claim about the hyperbolic metric follows readily from the definition of the Hilbert metric.
\end{proof}

\medskip
There are also geometric methods to construct the Hilbert midpoint. We refer the reader
to \cite[Fig. 10]{vw}.

\begin{lem} \label{}
For $x,y \in \mathbb{B}^2$
\[
|x-y| \le 2\, {\rm th} \frac{h_{ \mathbb{B}^2}(x,y)}{4}
\,.
\]
Equality holds if $x= -y\,.$
\end{lem}

\begin{proof} Let $z$ be the Hilbert midpoint of $x$ and $y$ and $t = h_{ \mathbb{B}^2}(x,y)/2\,.$ Then $x,y \in \partial B_h(z,t).$ From the formula  for the major semi-axis of the ellipse $ \partial B_h(z,t)$ in Remark \ref{semiaxes} we see that $ d( B_h(z,t))$ decreases as a function of $|z|$ and hence
\[
|x-y| \le d( B_h(z,t)) \le d( B_h(0,t)).
\]
Because $B_h(0,t) = B_{\rho}(0,t)= B^2(0, {\rm th } (t/2))$ by Theorem \ref{4.2} and \eqref{hkv420}, we have 
\[
|x-y| \le 2\, {\rm th} \frac{h_{ \mathbb{B}^2}(x,y)}{4}
\,.
\]
The sharpness follows from the formula for the hyperbolic metric.

\end{proof}

\medskip
Clearly, convex domains are preserved under affine mappings.
It is well-known that the Hilbert metric is preserved under affine mappings.
According to Theorem \ref{fujiEllip} we may conclude that Hilbert disks in ellipses
are also ellipses.

\section*{Acknowledgements}
The authors are deeply grateful to the anonymous referee for several constructive suggestions, insightful remarks, and careful reading, which greatly enhanced the quality and clarity of this manuscript.
\begin{itemize}
	\item Funding: Not applicable
	\item Conflict of interest: There is no conflict of interest
	\item Data availability: Not applicable 
\end{itemize}



\vspace{0.5cm}


\begin{thebibliography}{HIMPS}






\bibitem[B]{b} {\sc  A.\,F. Beardon}, The Geometry of Discrete Groups,
 Springer-Verlag, New York, 1983.
 \bibitem[B2]{b2} {\sc  A.\,F. Beardon},  The Klein, Hilbert, and Poincar\'e 
 metrics of a domain, J. Comput. Appl. Math., 105 (1999), 155-162.

 

\bibitem[DHV]{dhv} {\sc O. Dovgoshey, P. Hariri, and  M. Vuorinen},
{Comparison theorems for hyperbolic type metrics}, {Complex Var. Elliptic Equ.}
61,  11, (2016), 1464--1480.
 


 
 \bibitem[FKV]{fkv} {\sc M. Fujimura, R. Kargar, and M.  Vuorinen, } Formulas 
for the visual angle metric,  J. Geom. Anal. 
(2025), 35:322 (17 pages)
{arXiv:2304.04485}.
\bibitem[FRV]{frv}    {\sc M. Fujimura, O. Rainio, and  M. Vuorinen},
 {Collinearity of points on Poincar\'e unit disk and Riemann sphere},
  Publ.  Math. Debrecen 105 (2024) 1-2, 141--169, {arXiv:2212.09037}. 

\bibitem[GH]{gh}  {\sc F.W. Gehring and K. Hag},  The Ubiquitous Quasidisk.
With contributions by Ole Jacob Broch.
American Mathematical Society, Providence, RI, 2012. xii+171 pp.



\bibitem[HKV]{hkv}  {\sc P. Hariri, R. Kl\'en, and M. Vuorinen},
Conformally Invariant Metrics  and Quasiconformal Mappings,
Springer Monographs in Mathematics, Springer, Berlin, 2020.


\bibitem[HIMPS]{ha}{\sc P.A.  H\"ast\"o, Z. Ibragimov, D. Minda,
S. Ponnusamy, and S. Sahoo,} Isometries of some hyperbolic-type path metrics, 
and the hyperbolic medial axis.
(English summary) In the tradition of Ahlfors-Bers. IV, 63--74,
Contemp. Math., 432, Amer. Math. Soc., Providence, RI, 2007.



\bibitem[H]{h}{\sc J. Heinonen,} {Lectures on Analysis on Metric Spaces.}
 Springer-Verlag, New York, 2001.


\bibitem[HILB]{hilb}{\sc D. Hilbert,} Ueber die gerade Linie als 
k\"urzeste Verbindung zweier Punkte, Math. Ann. 46 (1895),  91-96.



\bibitem[LV]{lv}{\sc O. Lehto and K. I. Virtanen,} Quasiconformal mappings in the plane. Second edition. Translated from the German by K. W. Lucas. Die Grundlehren der mathematischen Wissenschaften, Band 126. Springer--Verlag, New York--Heidelberg, 1973.
 
 \bibitem[N]{n}{\sc M. Noro,} A computer algebra system Risa/Asir.  


\bibitem[P]{p} {\sc A. Papadopoulos, }  Metric spaces, convexity and non-positive curvature. Second edition. IRMA Lectures in Mathematics and Theoretical Physics,
 6. European Mathematical Society (EMS), Z\"{u}rich, 2014. xii+309 pp.
\bibitem[PT]{pt}{\sc A.  Papadopoulos and M. Troyanov,} 
From Funk to Hilbert geometry. Handbook of Hilbert geometry, 33--67, 
IRMA Lect. Math. Theor. Phys., 22, Eur. Math. Soc., Z\"urich, 2014.
\bibitem[PY1]{py1}{\sc A. Papadopoulos and S. Yamada,} 
The Funk and Hilbert geometries for spaces of constant curvature. 
Monatsh. Math. 172 (2013), no. 1, 97--120.

\bibitem[PY2]{py2}{\sc A. Papadopoulos and S. Yamada,} Funk and Hilbert 
geometries in spaces of constant curvature. Handbook of Hilbert geometry, 
353--379, IRMA Lect. Math. Theor. Phys., 22, Eur. Math. Soc., Z\"urich, 2014. 
%
  
\bibitem[RV]{rv}{\sc O. Rainio and M.  Vuorinen,} Hilbert metric in the unit ball, 
Studia Sci. Math. Hungar.  60(2-3), 2023, 175-191.

\bibitem[SOP]{sop}{\sc
I. Saglam, K. Ohshika, and A. Papadopoulos,}
On families of Finsler metrics.
Turk. J. Math. 50(5), 2026, 875-901.

\bibitem[VW]{vw}{\sc  M. Vuorinen and G. Wang,} Bisection of geodesic segments 
in hyperbolic geometry, Complex Analysis  and Dynamical Systems V, Contemp. Math., Israel Math. Conf. Proc., Amer. Math. Soc., Providence, RI. 591, 2013, 273--290.








\end{thebibliography}
\end{document}